\documentclass[12pt,a4paper, leqno]{article}
\usepackage{amssymb, amscd, amsthm, amsmath, amstext, colonequals}
\usepackage{enumitem}                                                          
\usepackage[all]{xy}
\usepackage{color}                                                             

\usepackage[colorlinks, linkcolor=blue, anchorcolor=green, citecolor=blue]{hyperref}

\newtheorem{thm}{Theorem}[section]
\newtheorem{coro}[thm]{Corollary}
\newtheorem{lemma}[thm]{Lemma}
\newtheorem{prop}[thm]{Proposition}
\theoremstyle{definition}
\newtheorem{remark}[thm]{Remark}

\theoremstyle{definition}
\newtheorem{defn}[thm]{Definition}
\newtheorem{example}[thm]{Example}

\numberwithin{equation}{thm}
\newcommand{\newpara}{\noindent\refstepcounter{thm}{\bf(\thethm)\;}}  

\newcommand{\car}{\mathrm{char}}
\newcommand{\cls}{\mathrm{cls}}
\newcommand{\gen}{\mathrm{gen}}

\newcommand{\Ker}{\mathrm{Ker}}
\newcommand{\rO}{\mathrm{O}}
\newcommand{\Pic}{\mathrm{Pic}}

\newcommand{\Spm}{\mathrm{Spm}}
\newcommand{\spn}{\mathrm{spn}}

\newcommand{\C}{\mathbb{C}}
\newcommand{\N}{\mathbb{N}}

\newcommand{\R}{\mathbb{R}}
\newcommand{\Z}{\mathbb{Z}}

\newcommand{\fp}{\mathfrak{p}}

\newcommand{\bfA}{\mathbf{A}}
\newcommand{\bfI}{\mathbf{I}}
\newcommand{\bfO}{\mathbf{O}}

\newcommand{\bfSpin}{\mathbf{Spin}}

\definecolor{tiangreen}{rgb}{0.24, 0.82, 0.44}                                 

\begin{document}

\title{\textbf{Strong Approximation and Hasse Principle for Integral Quadratic Forms over Affine Curves}}
\author{Yong HU, Jing LIU and Yisheng TIAN}
\date{}

\maketitle 

\begin{flushright}

\emph{In memory of Professor Linsheng YIN}

\end{flushright}

\begin{abstract}
We extend some parts of the representation theory for integral quadratic forms over the ring of integers of a number field to the case  over the coordinate ring $k[C]$ of an affine curve $C$ over a general base field $k$.
By using the genus theory, we link the strong approximation property of certain spin groups to the Hasse principle for representations of integral quadratic forms over $k[C]$ and derive several applications.
In particular, we give an example where a spin group does not satisfy strong approximation.
\end{abstract}



\section{Introduction}

As is well known, the foundation of an arithmetic theory of quadratic forms over the rational numbers started with Minkowski's work in the 1880s, and was  generalized to general number fields by Hasse in the 1920s. Witt's work \cite{Witt37} marks the beginning of the modern algebraic theory of quadratic forms over arbitrary fields. On the other hand, over rings of algebraic integers in number fields, there has also been a very well developed theory of integral quadratic forms for a long time, dating back at least to the work of Eichler and Kneser (e.g. \cite{Eichler52book}, \cite{Kneser56ArchMathVol7}). Among others, they introduced the notions of genera and spinor genera of quadratic forms in any number of variables, which play an essential role in the investigation of the Hasse principle for integral quadratic forms. In Kneser's work, strong approximation for spin groups has been used to study class numbers of integral forms. As tremendous subsequent research achievements (\cite{HsiaKitaokaKneser78crelle}, \cite{HsiaShaoXu98crelle}, \cite{ChanXu04compositio}, \cite{XuFei05MA}, \cite{XuZhang22TAMS}, etc.) demonstrate, results and ideas from Eichler and Kneser, in particular those about the genus theory,  have great influence on the subject.

Let $C$ be an irreducible normal affine curve over a field $k$ of characteristic different from 2. In various situations where $k$ is a field of arithmetic interest, a considerably large amount of research has been done on the arithmetic of quadratic forms over the function field $k(C)$ (rather than the coordinate ring $k[C]$). We mention in particular the papers \cite{HHK}, \cite{PS10}, \cite{CTPS12CMH} (where $k$ is a complete discrete valuation field) and \cite{LPS12} (where $k$ is a number field). As for integral quadratic forms, some techniques and results in the case where $k[C]$ is a one-variable polynomial ring (and $C$ is the affine line) are discussed in Gerstein's book \cite{Gerstein08book}. In fact, basic definitions such as genera and spinor genera can be easily generalized to the general case over arbitrary Dedekind domains, and hence apply to the affine coordinate ring $k[C]$ of a general curve. However, except when $k$ is finite (in which case $k(C)$ is a global field), we feel  that in the general setting, the integral theory over $k[C]$ has not been sufficiently developed, and that the strength of the genus theory  does not seem to have been taken full advantage of.

\medskip

In this paper, we first clarify several situations where the genus theory of integral quadratic forms (or quadratic lattices in geometric terms) over a general Dedekind domain $R$ can be useful. In \S\ref{sec2} we show that the proper spinor class number $g^+(L)$ of a lattice $L$ can be bounded by an idelic index associated with $L$, and that under some assumptions on local spinor norms, $g^+(L)$ can also be bounded by the size of a 2-torsion quotient of the Picard group of $R$. Then we prove that analogous to the global field case, proper spinor classes of $L$ coincide with its proper classes provided that the spin group of the quadratic space $V$ spanned by $L$ satisfies strong approximation over $R$ (Theorem\;\ref{3.2}). This result is used to deduce two sufficient conditions for the (integral)  Hasse principle for representations of quadratic lattices (Theorem\;\ref{3.5}). It is known that the spin group $\mathbf{Spin}(V)$ satisfies strong approximation when $V$ is isotropic. So, in that case we obtain several applications of the integral Hasse principle when $R=k[C]$ for a curve defined over a field $k$ with some special properties.

Finally, taking $R$ to be a polynomial ring over a $C_1$-field $k$ (in the sense of \cite{Lang52}), we give an example to show that when $V$ is anisotropic, the integral Hasse principle for the representability by $L$ and strong approximation of $\mathbf{Spin}(V)$ can fail (Theorem\;\ref{4.2}). In our example the group $\mathbf{Spin}(V)$ can also be viewed as a simply connected group of inner type A. So this shows that the isotropy assumption cannot be dropped in \cite[Theorem\;4.3]{HuTian23Whitehead}, where strong approximation is proved for isotropic groups of the same type (if $k$ is a $C_1$-field).

\section{Genus theory for quadratic lattices}\label{sec2}

As is customary, we shall use the language of lattices in the study of integral quadratic forms. A standard reference is O'Meara's book \cite{OMeara00}. In this section we recall some key definitions and facts, and we show how to extend the genus theory from the classical global field case to a more general setting.

\medskip

We fix a Dedekind domain $R$ with fraction field $F$.
Assume that the characteristic $\car(F)$ is not 2.
Let $\Spm(R)$ denote the set of maximal ideals of $R$.
Each $\fp\in\Spm(R)$ is also called a \emph{place} of $F$ over $R$.
For each place $\fp$, let $\widehat{R}_{\fp}$ be the completion of $R$ at $\fp$ and let $F_{\fp}$ be the fraction field of $\widehat{R}_{\fp}$.
If the residue field $\kappa(\fp)\colonequals R/\fp$ has characteristic different from 2, we say that $\fp$ is a \emph{non-dyadic place}.

\

\newpara\label{2.1}  By a (quadratic) \emph{lattice} over $R$ we mean a finitely generated torsion-free $R$-module $L$ together with a nondegenerate quadratic form $f$ defined on the vector space $FL=F\otimes_RL$. Put $V=FL$. We will also say that $L$ is an $R$-lattice on the quadratic space $(V,\,f)$. The \emph{dimension} $\dim L$ of $L$ is defined as the dimension of the space $V=FL$.

If $f(L)\subseteq R$, we say that the $R$-lattice $L$ is \emph{integral}. If moreover $L$ is free as an $R$-module, then by choosing a basis of $L$ over $R$ we may identify the pair $(L,\,f)$ with a homogeneous degree two polynomial with coefficients in $R$. In this case the lattice $L$ will also be called an \emph{integral quadratic form} over $R$.

Let $\bfO^+(V)$ denote the special orthogonal group of $(V,\,f)$ as an algebraic group (\cite[\S23, p.~348]{KMRT98}).
Let $\rO^+(V)$ be the group of $F$-points of $\bfO^+(V)$.
This group acts naturally on the set of lattices on $V$.
The \emph{proper class} $\cls^+(L)$ of $L$ is defined as the $\rO^+(V)$-orbit of $L$ under this action, namely,
\[
\cls^+(L)\colonequals \{\sigma(L)\,|\,\sigma\in\rO^+(V)\}\,.
\]

For each $\fp\in\Spm(R)$, put $V_{\fp}=F_{\fp}V$ and $L_{\fp}=\widehat{R}_{\fp}L$. Let $\rO^+_{\bfA}(V)$ denote the group of adelic points of $\bfO^+(V)$ over $R$, i.e.,
\begin{align*}
\rO^+_{\bfA}(V)
\colonequals\, & \bigg\{(\sigma_{\fp})\in\prod_{\fp\in\Spm(R)}\rO^+(V_{\fp})\,\Big|\, \sigma_{\fp}(L_{\fp})=L_{\fp} \text{ for almost all } \fp\bigg\}\,
\\
=\, & \bigg\{(\sigma_{\fp})\in\prod_{\fp\in\Spm(R)}\rO^+(V_{\fp})\,\Big|\, ||\sigma_{\fp}||_{\fp}=1 \text{ for almost all } \fp\bigg\}\,,
\end{align*}
%
where $||\cdot||_{\fp}$ is the norm on $\rO^+(V_{\fp})$ induced by the $\fp$-adic norm on $F_{\fp}$.
The second equality follows from \cite[Example 101:1]{OMeara00}.
In particular, $\rO^+_{\bfA}(V)$ is independent of the choice of the $R$-lattice $L$ on $V$.
Moreover, via the diagonal embedding we may view $\rO^+(V)$ as a subgroup of $\rO^+_{\bfA}(V)$ (by \cite[Example 101:4]{OMeara00}).

For every $\underline{\sigma}=(\sigma_{\fp})\in\rO^+_{\bfA}(V)$,
the same argument as in \cite[\S\,101.D, p.~297]{OMeara00} shows that there is a unique lattice $K$ on $V$ such that $K_{\fp}=\sigma_{\fp}L_{\fp}$ for all $\fp\in\Spm(R)$. We denote this lattice $K$ by $\underline{\sigma}L$. The \emph{genus} $\gen(L)$ of $L$ is defined as the orbit of $L$ under this action of $\rO^+_{\bfA}(V)$, i.e.,
\[
\gen(L)\colonequals \{\underline{\sigma}L\,|\,\underline{\sigma}\in\rO^+_{\bfA}(V)\}\,.
\]One can decompose $\gen(L)$ into a disjoint union of proper classes $\cls^+(M)$ for a certain family of lattices $M$ on $V$. The number of proper classes in such a decomposition, denoted by $h^+(L)$ and considered as a positive integer or $+\infty$, is called the \emph{proper class number} of $L$.

\medskip

For each $\fp\in\Spm(R)$, let $\theta_{\fp}: \rO^+(V_{\fp})\to F^\times_{\fp}/F_{\fp}^{\times 2}$ be the spinor norm map for the quadratic space $V_{\fp}$ (\cite[\S55]{OMeara00}). (For a commutative ring $A$, we denote by $A^{\times}$ the group of units in $A$.)
For any subset $X\subseteq \rO^+(V_{\fp})$, we shall often consider the image $\theta_{\fp}(X)$ as a subset of $F_{\fp}^\times$, which is a union of cosets of $F_{\fp}^{\times 2}$. If $\fp$ is a non-dyadic place and the local lattice $L_{\fp}$ is modular (in the sense of \cite[\S\,82.G]{OMeara00}), then as in \cite[(92:5)]{OMeara00} one can show that
$\theta_{\fp}(\rO^+(L_{\fp}))\subseteq \widehat{R}_{\fp}^\times F_{\fp}^{\times 2}$,
where
\[
\rO^+(L_{\fp})\colonequals \{\sigma \in\rO^+(V_{\fp})\,|\,\sigma (L_{\fp})=L_{\fp}\}\,.
\]
Since $\car(F)\neq 2$, all but finitely many places in $\Spm(R)$ are non-dyadic. Thus, letting $\bfI_F$ be the \emph{id\`ele group} of $F$ with respect to $\Spm(R)$, i.e., the restricted product of $F_{\fp}^\times$ relative to $\widehat{R}_{\fp}^\times$ for $\fp\in\Spm(R)$, the local spinor norm maps induce an adelic spinor norm map $\theta_{\bfA}\,:\;\rO^+_{\bfA}(V)\to \bfI_F/\bfI_F^2$. We define
\[
\rO'_{\bfA}(V)\colonequals \Ker\big(\theta_{\bfA}\,:\;\rO^+_{\bfA}(V)\longrightarrow \bfI_F/\bfI_F^2\big)\,.
\]
The \emph{proper spinor genus} $\spn^+(L)$ of $L$ is defined to be the orbit of $L$ under the subgroup $\rO^+(V)\rO'_{\bfA}(V)$ of $\rO^+_{\bfA}(V)$.
By construction, we have
$\cls^+(L)\subseteq \spn^+(L)\subseteq\gen(L)$.
The \emph{proper spinor class number} $g^+(L)\in\N\cup\{+\infty\}$ is defined as the number of disjoint proper spinor genera contained in $\gen(L)$.

\begin{prop}\label{2.2}
  Let $L$ be an $R$-lattice and let $V=FL$.
  Define
  \[
  \bfI_F^L\colonequals \{(a_{\fp})\in\bfI_F\,|\, a_{\fp}\in \theta_{\fp}\big(\rO^+(L_{\fp})\big) \text{ for all } \fp\in\Spm(R)\}\,.
  \]Then $g^+(L)$ divides $\big[\bfI_F:\theta_{\bfA}(\rO^+(V))\bfI_F^L\big]$ if the latter is finite.
\end{prop}
\begin{proof}
The proof of \cite[(102:7a)]{OMeara00} is still valid in our situation.
Indeed, the adelic spinor norm induces a group homomorphism
$\phi:\rO^+_{\bfA}(V)\to \bfI_F/\big(\theta_{\bfA}(\rO^+(V))\bfI_F^L\big)$ with $g^+(L)=[\rO^+_{\bfA}(V):\Ker(\phi)]$.
\end{proof}

Let $\Pic(R)$ denote the Picard group of $R$ and put $\Pic(R)/2\colonequals \Pic(R)\otimes_{\Z}(\Z/2\Z)$.
\begin{prop}\label{2.3}
  Let $L$ be an  $R$-lattice and $V=FL$. Suppose that the following conditions hold:

\begin{enumerate}[label={{\upshape(\roman*)}}]
\item\label{lattice dividing property 1}
The spinor norm map $\theta: \rO^+(V)\to F^\times/F^{\times 2}$ is surjective.

\item\label{lattice dividing property 2}
For all $\fp\in\Spm(R)$, $\theta_{\fp}(\rO^+(L_{\fp}))\supseteq \widehat{R}_{\fp}^\times$.
\end{enumerate}

  Then $g^+(L)$ divides $\#\big(\Pic(R)/2\big)$ if the latter is finite. In particular, we have $g^+(L)=1$ if $\Pic(R)/2=0$ (e.g. if $R$ is a principal ideal domain).
\end{prop}
\begin{proof}
Putting $\mathbf{U}_F\colonequals \prod_{\fp\in\Spm(R)}\widehat{R}_{\fp}^\times$, we have an isomorphism  $\bfI_F/(F^\times\mathbf{U}_F)\cong \Pic(R)$, whence
$\bfI_F\big/\big(F^\times\mathbf{U}_F\bfI_F^2\big)\cong \Pic(R)/2$.  Condition (ii) implies $\mathbf{U}_F\bfI_F^2\subseteq\bfI_F^L$. Thus, condition (i) combined with Proposition\;\ref{2.2} shows that $g^+(L)\,|\,\#\big(\Pic(R)/2\big)$.
\end{proof}

\begin{example}\label{2.4}
Let $k$ be a field of characteristic $\neq 2$.
Let $R=k[C]$ be the coordinate ring of an irreducible normal affine curve $C$ defined over $k$.
Without loss of generality, we may assume that $k$ is algebraically closed in $R$.
Let $\widetilde{C}$ be an irreducible normal projective curve containing $C$ as a dense open subset, and let $S=\widetilde{C}\setminus C$.

  Let us give some examples where $\Pic(R)/2$ is finite or even trivial.

\begin{enumerate}[label={{\upshape(\arabic*)}}]
  \item Suppose that $\widetilde{C}$  is a smooth plane conic. Then $\Pic(\widetilde{C})/2\cong \Z/2\Z$ and hence $\#(\Pic(R)/2)\le 2$.

  If  either $S$   contains a $k$-rational point or $S$ contains a closed point of degree 2 and $\widetilde{C}$ has no $k$-rational points, then $\Pic(R)=0$ (\cite[Thm.\;5.1]{Samuel61Illinois}).
  \item Suppose that $\widetilde{C}$ is smooth and geometrically connected, and that  $k$ is finitely generated (over its prime field). Then $\Pic(\widetilde{C})$ is finitely generated and hence $\Pic(R)/2$ is finite (\cite[Thm.\;1.9]{CT93LNM1553}).

       If $S$ contains a system of generators of $\Pic(\widetilde{C})$, then $\Pic(R)=0$.
  \item Suppose that $k$ has \emph{finite cohomology} at the prime $2$, i.e., for every finite 2-primary torsion Galois module $\mu$ over $k$, the Galois cohomology groups $H^r(k,\,\mu),\,r\in\N$ are all finite. For example, $k$ can be an iterated Laurent series field of the form $k_0(\!(t_1)\!)\cdots(\!(t_n)\!)$, where $k_0$ is $\R,\,\C$, a $p$-adic field or a finite field.

      In this case $\Pic(\widetilde{C})/2$ and $\Pic(R)/2$ are finite. Indeed, if $\bar k$ denotes a separable closure of $k$, it is a standard fact (see e.g. \cite[Chap.\;I, Prop.\;5.3]{FK}) that the \'etale cohomology groups       $H^r(\widetilde{C}\times_k\bar k\,,\,\Z/2),\,r\in\N$ are all finite. Using the Hochschild--Serre spectral sequence we can deduce the finiteness of $H^2(\widetilde{C},\,\Z/2)$. Then the injection $\Pic(\widetilde{C})/2 \hookrightarrow H^2(\widetilde{C},\,\Z/2)$ from the Kummer theory gives the desired result.
\end{enumerate}
\end{example}

\section{Strong approximation and integral Hasse principle}
Throughout this section, $R$ is a Dedekind domain with fraction field $F$. Assume $\car(F)\linebreak\neq 2$.

Our goal in this section is to relate the strong approximation property of certain homogeneous varieties to a Hasse principle for representations of quadratic lattices (cf. Theorem\;\ref{3.5}). We also discuss some applications of the Hasse principle to integral quadratic forms.

\medskip

\newpara\label{3.1}
Let $X$ be a variety over $F$. For every field extension $E/F$, let $X(E)$ denote the set of $E$-points of $X$.
For each $\fp\in\Spm(R)$, the $\fp$-adic topology of $F_{\fp}$ induces a topology on $X(F_{\fp})$.
We define the \emph{adelic space} $X(\bfA_F)$ (over $R$) as
the restricted topological product space of the $\fp$-adic topological spaces $X(F_{\fp}),\,\fp\in\Spm(R)$
with respect to the open subspaces $X(\widehat{R}_{\fp})$ defined for all but finitely many $\fp\in\Spm(R)$.
We call the resulting topology on $X(\bfA_F)$ the adelic topology.

We say that $X$ satisfies \emph{strong approximation} (over $R$) if the diagonal image of $X(F)$ is dense in $X(\bfA_F)$ with respect to the adelic topology.
Explicitly, this means that for every (large enough) finite subset $\Sigma\subseteq\Spm(R)$, when an arbitrary nonempty $\fp$-adic open subset $U_{\fp}$ of $X(F_{\fp})$ is given for each $\fp\in \Sigma$,
 the diagonal image of $X(F)$ contains a point in $\prod_{\fp\in\Sigma}U_{\fp}\times\prod_{\fp\notin\Sigma}X(\widehat{R}_{\fp})$.

Assume further that $X$ admits an integral model over $R$,
i.e., a flat separated integral $R$-scheme $\mathcal{X}$ of finite type such that $\mathcal{X}\times_RF\cong X$.
Then, by considering the intersection of $X(F)$ with the product space $\prod_{\fp\in\Spm(R)}X(\widehat{R}_{\fp})$,
we see that $\mathcal{X}$ has an $R$-point if $X$ satisfies strong approximation and $X(\widehat{R}_{\fp})\neq\emptyset$ for all $\fp\in\Spm(R)$.

\

In the remainder of this section,
we fix  a nondegenerate quadratic space  $(V,\,f)$  of dimension $\ge 3$ over $F$.
Let $G\colonequals \bfSpin(V)$ be its spin group as an algebraic group (\cite[\S23, pp.~349-352]{KMRT98}).
Let $L$ be an integral $R$-lattice on $V$.

\medskip

The following theorem is an analogue of a well known result over global fields (cf. \cite[(104:5)]{OMeara00}).
We provide some details in the proof for the reader's convenience.

\begin{thm}\label{3.2}
Suppose that $G\colonequals \bfSpin(V)$ satisfies strong approximation over $R$. Then $\cls^+(L)=\spn^+(L)$ and hence $h^+(L)=g^+(L)$.
\end{thm}
\begin{proof}
Since the first assertion also applies to any other lattice on $V$, the equality $h^+(L)=g^+(L)$ follows from their definitions.
So it suffices to prove $\cls^+(L)=\spn^+(L)$.
Consider the following exact sequence of algebraic groups
\[
1\longrightarrow \mu_2\longrightarrow G\longrightarrow \bfO^+(V)\longrightarrow 1\,.
\]
The connecting map $\bfO^+(V)(F)\to H^1(F,\,\mu_2)$ in the associated exact sequence of Galois cohomology sets coincides with the spinor norm map
$\theta: \rO^+(V)\to F^\times/F^{\times 2}$.
The same holds over each local completion $F_{\fp}$ of $F$.
So we have a commutative diagram with exact rows
\begin{equation}\label{eq3.2.1}
\begin{array}{c}
\xymatrix{
G(F) \ar[d]\ar[r] & \rO^+(V)\ar[d] \ar[rr]^{\theta} && F^\times/F^{\times 2} \ar[d]\\
G(\bfA_F) \ar[r]^{\pi} & \rO^+_{\bfA}(V) \ar[rr]^{\theta_{\bfA}} && \bfI_F/\bfI_F^2
}
\end{array}
\end{equation}

Consider a lattice $K\in\spn^+(L)$. By definition, there exists $\sigma\in\rO^+(V)$ and $\underline{\gamma}=(\gamma_{\fp})\in \rO'_{\bfA}(V)=\mathrm{Ker}(\theta_{\bfA})$ such that
$K=\sigma\underline{\gamma}L$. We want to show $K\in \cls^+(L)$. Since $\sigma^{-1}K$ is in the same proper class as $K$, we may assume $\sigma$ is the identity.

Note that $\underline{\gamma}$ lies in the image of the map $\pi: G(\bfA_F)\to \rO^+_{\bfA}(V)$. Thus, using the commutative diagram \eqref{eq3.2.1} and the strong approximation property of $G$, we can find an element $\rho\in \rO^+(V)\cap \underline{\gamma}\cdot\big(\prod_{\fp}\rO^+(L_{\fp})\big)$. Then $K=\underline{\gamma}L=\rho L\in\cls^+(L)$. This proves the theorem.
\end{proof}

\begin{coro}\label{3.3}
  Suppose that $G=\bfSpin(V)$ satisfies strong approximation over $R$, and that $L$ satisfies conditions \ref{lattice dividing property 1} and \ref{lattice dividing property 2} in \textup{Proposition}\;$\ref{2.3}$.

  Then $h^+(L)$ divides $\#(\Pic(R)/2)$ if the latter is finite $($cf. \textup{Example}$\;\ref{2.4})$.

  In particular, $h^+(L)=1$ if $R$ is a principal ideal domain.
\end{coro}
\begin{proof}
  Immediate from Theorem\;\ref{3.2} and Proposition\;\ref{2.3}.
\end{proof}

Recall that we have fixed a nondegenerate quadratic space  $(V,\,f)$  of dimension $\ge 3$ over $F$ and
an integral $R$-lattice $L$ on $V$.

\begin{defn}\label{3.4}
  Let $M$ be an $R$-lattice with associated quadratic form $q: FM\to F$.
  For any ring $R'$ (which need not be an integral domain) containing $R$,
  we say that $M$ \emph{is represented} by $L$ over $R'$
  if there is an $R'$-module homomorphism $\sigma: M\otimes_RR'\to L\otimes_RR'$ such that $f\circ\sigma=q$.

  We say that the representability  $M\to L$ satisfies the \emph{Hasse principle} if $M$ is represented by $L$ over $R$ as long as it is represented by $L$ over $F$ and over $\widehat{R}_{\fp}$ for every $\fp\in\Spm(R)$.

A nonzero element $a\in R$ is called \emph{represented} by $L$ over $R'$ if the free lattice $Rv$ with $Q(v)=a$ is represented by $L$ over $R'$.
\end{defn}

\begin{thm}\label{3.5}
  For any $R$-lattice $M$, the representability $M\to L$ satisfies the Hasse principle in each of the following two cases:

\begin{enumerate}
\item[\upshape(i)]
$($Class number criterion$)$ The proper class number $h^+(L)$ is $1$.

\item[\upshape(ii)]
$($Integral point criterion$)$ The group $G=\bfSpin(V)$ satisfies strong approximation over $R$, the field $F$ has cohomological $2$-dimension $\mathrm{cd}_2(F)\le 2$, and $\dim L\ge \dim M+3$.
\end{enumerate}
\end{thm}
\begin{proof}
\hfill
\begin{enumerate}
\item[\upshape(i)]
If $M$ is represented by $L$ over $F$ and over every local completion $\widehat{R}_{\fp}$, then $M$ is represented by some lattice $L'$ in the genus of $L$ (cf. \cite[(102:5)]{OMeara00}). The assumption $h^+(L)=1$ implies that $\gen(L)=\cls^+(L)$, whence $L'\cong L$.

\item[\upshape(ii)]
Let $(W,\,q)$ denote the quadratic space spanned by $M$. Assuming  $M$  represented by $L$ over $F$ means that $(W,\,q)$ can be embedded as a subspace of $(V,\,f)$. By the Witt theory of quadratic spaces over fields, we can find an orthogonal decomposition $(V,\,f)\cong (W,\,q)\perp (U,\,\varphi)$. Let $H$ be the spin group of the space $(U,\,\varphi)$ and $X\colonequals G/H$. As explained in \cite[\S\S\,5 and 7]{ColliotXu09Compositio}, the $F$-variety $X$ has an integral model $\mathcal{X}$ over $R$ with the property that $M$ is represented by $L$ over $R$ if and only if $\mathcal{X}(R)\neq\emptyset$. Suppose that $M$ is represented by $L$ over $\widehat{R}_{\fp}$ for all $\fp\in\Spm(R)$. Then, as we have said in \eqref{3.1}, $M$ is represented by $L$ over $R$ if $X$ satisfies strong approximation over $R$.

When $\dim U=\dim V-\dim W\ge 3$, the group $H$ is semisimple simply connected.
The assumption $\mathrm{cd}_2(F)\le 2$ implies that $H^1(F_{\fp},\,H)=1$ for all $\fp\in\Spm(R)$ (cf. \cite{BP1}).
A diagram chase as in the proof of \cite[Thm.\;3.4]{Colliot18EurJmath} shows that $X$ satisfies strong approximation over $R$ as $G$ does.
This completes the proof.
\qedhere
\end{enumerate}
\end{proof}

\begin{remark}\label{3.6}
The integral point criterion in Theorem\;\ref{3.5} (ii) has a variant as follows.
\emph{The group $G=\bfSpin(V)$ satisfies strong approximation over $R$, $\dim L=\dim M+2$, and $-\det(V)\det(FM)$ is a square in $F$.}

Indeed, in the above case the group $H$ is isomorphic to the split torus $\mathbb{G}_m$ and hence  $H^1(F_{\fp},\,H)=1$ for all $\fp\in\Spm(R)$.
\end{remark}

\begin{coro}\label{3.7}
  Suppose that the quadratic space $(V,\,f)$ is isotropic.

  \begin{enumerate}[label={{\upshape(\arabic*)}}]
    \item Suppose that $L$ is locally universal at every $\fp\in\Spm(\widehat{R}_{\fp})$, i.e., $L_{\fp}$ represents all the nonzero elements of $\widehat{R}_{\fp}$.

    Then $h^+(L)$ divides $\#(\Pic(R)/2)$ if the latter is finite $($cf. Example$\;\ref{2.4})$.

    If $R$ is a principal ideal domain, then for every $R$-lattice $M$, the representability $M\to L$ satisfies the Hasse principle.
    \item Suppose that $\mathrm{cd}_2(F)\le 2$. Then for every $R$-lattice $M$ with $\dim M\le\dim L-3$, the representability $M\to L$ satisfies the Hasse principle.
  \end{enumerate}
\end{coro}
\begin{proof}
  The spinor norm map $\theta: \rO^+(V)\to F^\times /F^{\times 2}$ is surjective since $f$ is isotropic. On the other hand, the group $G=\bfSpin(V)$ is $F$-rational by \cite{Platonov79Dokl}. Hence $G$ satisfies strong approximation over $R$ by \cite[Cor.\;5.11]{Gille09SemBourbaki}.

  Assertion (2) is thus an immediate consequence of the integral point criterion in Theorem\;\ref{3.5}.

  To prove (1), suppose that $L$ is locally universal. Let $\fp\in\Spm(R)$ and $\alpha\in \widehat{R}_{\fp}^{\times}$. Then there exist vectors $x,\,y\in L_{\fp}$ such that $f(x)=1$ and $f(y)=\alpha$. The product of the two symmetries $\tau_x$ and $\tau_y$ belongs to the subgroup $\rO^+(L_{\fp})$ and we have $\theta_{\fp}(\tau_x\tau_y)=\alpha$. So Condition (ii) of Proposition\;\ref{2.3} holds.

  The first assertion now follows from Corollary\;\ref{3.3}, and the second assertion holds by the class number criterion in Theorem\;\ref{3.5}.
\end{proof}

\begin{coro}\label{3.8}
Let $k$ be a finite extension of $\mathbb{R}(t)$ or of $\mathbb{R}(\!(t)\!)$.
Let $R=k[C]$ be the coordinate ring of an irreducible normal affine curve $C$ over $k$.
Suppose that the function field $F=k(C)$ is nonreal $($i.e., $-1$ is a sum of finitely many squares in $F)$.

\begin{enumerate}[label={{\upshape(\arabic*)}}]
  \item Every integral quadratic form $g$ over $R$ in $n\ge 2$ variables can be written as a sum of $n+3$ squares of linear forms over $R$ $($in the same set of variables as $g)$.
  \item Assume further that $R $ is a principal ideal domain. Let $L$ be an integral $R$-lattice of dimension $\ge 5$.

   Then $L$ is globally universal over $R$ $($i.e., represents all the nonzero elements of $R)$ if and only if it is locally universal.
\end{enumerate}
\end{coro}
\begin{proof}
The field $F$ has $u$-invariant $\le 4$, that is, every quadratic form in at least 5 variables over $F$ is isotropic. If $k$ is a finite extension of $\mathbb{R}(\!(t)\!)$, this can be deduced from the Hasse principle given in \cite[Thm.\;3.1]{CTPS12CMH}. If $k$ is a finite extension of $\mathbb{R}(t)$, this was proved in \cite[Thm.\;0.10]{Benoist19PIHES}.

Assertion (2) is thus immediate from Corollary\;\ref{3.7} (1).

To prove (1), let $I_{n+3}$ denote the integral $R$-lattice defined on the $R$-module $R^{\oplus (n+3)}$ by  the quadratic form $x_1^2+\cdots+x_{n+3}^2$. By Corollary\;\ref{3.7} (2), the representability $g\to I_{n+3}$ satisfies the Hasse principle. So it is sufficient to check that $g$ is represented by $I_{n+3}$ over $\widehat{R}_{\fp}$ for all $\fp\in\Spm(R)$. By \cite[p.~850, Thm.\;1]{OMeara58AJM} (whose proof does not rely on the finite residue field assumption), this follows from the fact that the space $F_{\fp}I_{n+3}$ represents all $n$-ary spaces over $F_{\fp}$.
\end{proof}

According to Mordell \cite{Mordell30} and Ko \cite{Ko37}, Corollary\;\ref{3.8} (1) has the following analogue: \emph{For $2\le n\le 5$, every positive definite quadratic form over $\mathbb{Z}$ in $n$ variables can be written as a sum of $n+3$ squares of integral linear forms. }

\begin{coro}\label{3.9}
Let $k$ be a number field or a $p$-adic field.
Let $R=k[C]$ be the coordinate ring of an irreducible normal affine curve $C$ over $k$.
Suppose that the function field $F=k(C)$ is nonreal $($which is always the case if $k$ is $p$-adic$)$.

\begin{enumerate}[label={{\upshape(\arabic*)}}]
\item
For every integer $m\ge 5$, the $R$-lattice $I_m$ has a finite proper class number $h^+(I_m)$.

\item
Suppose that $R$ is a principal ideal domain.
Every integral quadratic form in $n\ge 1$ variables over $R$ can be written as a sum of $n+7$ squares of linear forms over $R$.
\end{enumerate}
\end{coro}
\begin{proof}
We only consider the case where $k$ is a number field, the $p$-adic case being completely analogous.

\begin{enumerate}[label={{\upshape(\arabic*)}}]
\item
For $m\ge 5$, the  quadratic form $I_m$ is isotropic over $F$. For every $\fp\in\Spm(R)$, the completion $F_{\fp}$ is isomorphic to $k'(\!(t)\!)$, where $k'$ is a finite extension of $k$. Since $F$ is nonreal, the field $F_{\fp}$ is also nonreal. This means that $k'$ is a nonreal number field. This implies that over $\widehat{R}_{\fp}\cong k'[\![t]\!]$ the lattice $I_m$ represents all the elements of $\widehat{R}_{\fp}$. Hence, by Corollary\;\ref{3.7} (1) we have $h^+(I_m)\,|\,\#(\Pic(R)/2)$ (and we have seen the finiteness of $\#(\Pic(R)/2)$ in Example\;\ref{2.4}).

\item
Notice that for each $n\ge 1$, the representability by the lattice $I_{n+7}$ satisfies the Hasse principle by Corollary\;\ref{3.7} (1). For each $\fp\in\Spm(R)$, the local completion $F_{\fp}$ has $u$-invariant $\le 8$. So the quadratic space $F_{\fp}I_{n+7}$ represents all $n$-ary quadratic spaces over $F_{\fp}$. As in the proof of Corollary\;\ref{3.8} (1), the result now follows by applying \cite[p.~850, Thm.\;1]{OMeara58AJM}.
\qedhere
\end{enumerate}
\end{proof}

\section{A counterexample}\label{sec5}

In this section, we give an example where the Hasse principle for representations of integral quadratic forms fails. As a consequence, we also obtain a counterexample where a spin group does not satisfy strong approximation.

\

We fix a field $k$ of characteristic $\neq 2$.  Let $g\in k[x]$ be a monic polynomial of even degree and let $t\in k^*$ be a non-square element. Put $a_1=1$, $a_2=-t$, $a_3=xt$ and $a_4=-xg$.

\begin{lemma}\label{4.1}
 Let $\alpha_1,\cdots,\alpha_4\in k[x]$, and put $\varphi=\sum^4_{i=1}a_i\alpha_i^2$.

\begin{enumerate}[label={{\upshape(\arabic*)}}]
    \item We have $\deg(\varphi)=\max\limits_{1\le i\le 4}\{\deg(a_i\alpha_i^2)\}$.

    \item If $\deg(g)>1$, then $\varphi\neq x$.
\end{enumerate}
\end{lemma}
\begin{proof}
Without loss of generality, we may assume that the $\alpha_i$ are not all 0. Put $M=\max\limits_{1\le i\le 4}\{\deg(a_i\alpha_i^2)\}\ge 0$ and  let $\varepsilon_i\in k$  be the leading coefficient of $\alpha_i$ for each $i$.
\begin{enumerate}[label={{\upshape(\arabic*)}}]
\item
To show that $\deg(\varphi)=M$, we may assume that $\deg(a_i\alpha_i^2)=M$ for at least two indices $i\in \{1,2,3,4\}$.
By construction of the $a_i$,
$\deg(a_1\alpha_1^2)$ and $\deg(a_2\alpha_2^2)$ are even,
while $\deg(a_3\alpha_3^2)$ and $\deg(a_4\alpha_4^2)$ are odd.
So we may assume that either $M=\deg(a_1\alpha_1^2)=\deg(a_2\alpha_2^2)$ or $M=\deg(a_3\alpha_3^2)=\deg(a_4\alpha_4^2)$.

If $M=\deg(a_1\alpha_1^2)=\deg(a_2\alpha_2^2)$, then $\varepsilon_1\varepsilon_2\neq 0$, and since $t$ is not a square in $k$, we have
\[
a_1\varepsilon^2_1+a_2\varepsilon_2^2=\varepsilon^2_1-t\varepsilon_2^2\neq 0\,.
\]
So the leading term of $\varphi$ is $(\varepsilon^2_1-t\varepsilon_2^2)x^M$ in this case. Similarly, if $M=\deg(a_3\alpha_3^2)=\deg(a_4\alpha_4^2)$, then the leading term of $\varphi$ is $(t\varepsilon^2_3-\varepsilon_4^2)x^M$. The equality $\deg(\varphi)=M$ is thus proved.

\item
If $\deg(g)>1$, then $\deg(a_4g\alpha_4^2)>1$ for all nonzero $\alpha_4\in k[x]$. Hence, by (1), $\varphi=x$ can only happen when $\alpha_4=0$. Thus, the equality $\varphi=x$ forces $t\varepsilon_3^2x$ to be the leading term of $\varphi$. But this term cannot be $x$ since $t$ is not a square in $k$.
\qedhere
\end{enumerate}
\end{proof}

Recall that a field $k$ is called a $C_1$-field (\cite{Lang52}) if for every positive integer $n$, every homogeneous polynomial of degree $<n$ in $k[x_1,\cdots, x_n]$ has a nontrivial zero over $k$.
Recall that $g\in k[x]$ is a monic polynomial of even degree and let $t\in k^*$ be a non-square element.
We put $a_1=1$, $a_2=-t$, $a_3=xt$ and $a_4=-xg$.

\begin{thm}\label{4.2}
  With notation as above, assume further that $k$ is a $C_1$-field $($e.g. $k=\C(\!(t)\!))$.
  Put $R=k[x]$ and define the diagonal $R$-lattice $L=\langle a_1,a_2,a_3,a_4\rangle$, i.e.,
  \[
  L\colonequals Rv_1\perp Rv_2\perp Rv_3\perp Rv_4
  \]
  with the associated quadratic form $Q$ determined by $Q(v_i)=a_i$ for each $i=1,\cdots, 4$.

Suppose that $\deg(g)>1$ and that $g(0)=1$.

  \begin{enumerate}[label={{\upshape(\arabic*)}}]
    \item Over the rational function field $F=k(x)$ the quadratic space $V\colonequals FL$ is universal, i.e., represents all the nonzero elements of $F$.
    \item The $R$-lattice $L$ is locally universal at every $\fp\in\Spm(R)$.
    \item The spin group $\mathbf{Spin}(V)$ does not satisfy strong approximation over $R$.
  \end{enumerate}
\end{thm}
\begin{proof}
\hfill
\begin{enumerate}[label={{\upshape(\arabic*)}}]
  \item Since $k$ is a $C_1$-field, every quadratic form in more than 4 variables over $F=k(x)$ is isotropic. This implies that the quaternary space $FL$ is universal.

  \item Recall that a lattice over $\widehat{R}_{\fp}$ is unimodular if it is isomorphic to a lattice in diagonal form $\langle u_1,\cdots, u_r\rangle$ where each $u_i$ is a unit in  $\widehat{R}_{\fp}$. If $L'$ is a unimodular lattice over $\widehat{R}_{\fp}$, by \cite[p.~850, Thm.\;1]{OMeara58AJM}, an element $a\in \widehat{R}_{\fp}$ is represented by $L'$ if the quadratic space $F_{\fp}L'$ represents $a$ over $F_{\fp}$.

 If $\fp\neq xR$, then $a_1=1,\,a_2=-t$ and $a_3=xt$ are all units in $\widehat{R}_{\fp}$. This means that $L_1\colonequals \widehat{R}_{\fp}v_1\perp \widehat{R}_{\fp}v_2\perp \widehat{R}_{\fp}v_3=\langle a_1,a_2,a_3\rangle$ is a unimodular sublattice of $L_{\fp}$.  The residue field $\kappa(\fp)$ of $\widehat{R}_{\fp}$ is a finite extension of $k$, and hence a $C_1$-field. Therefore, ternary quadratic forms over $\kappa(\fp)$ are all isotropic. Using Hensel's lemma, we can deduce that the quadratic space $F_{\fp}L_1$ is isotropic and hence universal. As we have said previously, this shows that $L_1$ is universal. Thus $L_{\fp}$ is universal as well.

 Now suppose $\fp=xR$. The assumption $g(0)=1$ implies that  $g$ is a square in $\widehat{R}_{\fp}^\times$. Thus,
 \[
 L_{\fp}=\langle a_1,a_2\rangle\perp \langle a_3,a_4\rangle=\langle 1,\,-t\rangle\perp \langle xt,\,-xg\rangle\cong \langle 1,\,-t\rangle\perp \langle xt,\,-x\rangle\,.
 \]By Hensel's lemma, the binary lattice $\langle 1,\,-t\rangle$ represents all the elements in $\widehat{R}_{\fp}^\times$. This also implies that $\langle xt,\,-x\rangle$ represents all the elements in $x\widehat{R}_{\fp}^\times$. Hence $L_{\fp}$ represents all the elements in the union $\widehat{R}_{\fp}^\times\cup x\widehat{R}_{\fp}^\times$. Since $x$ is a uniformizer at $\fp=xR$, this shows that $L_{\fp}$ is universal over $\widehat{R}_{\fp}$.

  \item Combining (1), (2) and Lemma \ref{4.1}, we see that the representability of $x$ by $L$ does not satisfy the Hasse principle. So the result follows by the integral point criterion in Theorem \ref{3.5}.
\qedhere
\end{enumerate}
\end{proof}

\begin{remark}\label{4.3}
  The quadratic space  $V=FL$ in Theorem \ref{4.2} is \emph{definite}, in the sense that it is anisotropic over the completion $F_{\infty}=k(\!(x^{-1})\!)$. So there is no contradiction between Theorem\;\ref{4.2} and Corollary \ref{3.7}.

We think that whether there is an indefinite space $V$ over $F=k(x)$ such that strong approximation fails for $\mathbf{Spin}(V)$ is an interesting open question.
If $k$ is finite, so that $k(x)$ is a global field, it is well known that such an indefinite space does not exist (see e.g. \cite{Eichler52book} or \cite{Kneser66StrongApp}).
\end{remark}

\begin{remark}\label{4.4}
  The spin group $G=\mathbf{Spin}(V)$ in Theorem \ref{4.2} is a semisimple simply connected group of inner type A. Indeed, if $g$ is a square in $F=k(x)$,
  then the even Clifford algebra $C_0(V)$ of $V$ is isomorphic to $A\times A$ for some quaternion $F$-algebra $A$. In this case we have  $G\cong \mathbf{SL}_1(A)\times \mathbf{SL}_1(A)$. If $g$ is not a square in $F$, then $C_0(V)$ is a quaternion algebra over the quadratic extension $L\colonequals F(\sqrt{g})$. In this case $G$ is isomorphic to $R_{L/F}\mathbf{SL}_1(C_0(V))$, the Weil restriction of the $L$-group $\mathbf{SL}_1(C_0(V))$.

  In \cite[Thm.\;4.3]{HuTian23Whitehead}, it is proved that if $C$ is an irreducible normal affine over a field $k$ of characteristic 0 and cohomological dimension $\mathrm{cd}(k)\le 2$,
  then any central division algebra $D$ over the function field $k(C)$, the group $\mathbf{SL}_n(D)$ satisfies strong approximation over the coordinate ring $k[C]$ if $n\ge 2$. By our Theorem\;\ref{4.2},  that theorem fails if $n=1$.
\end{remark}

\subsection*{Acknowledgements}
The first author was supported by grants from the National Natural Science Foundation of China (no.\,12171223) and the Guangdong Basic and Applied Basic Research Foundation (no.\,2021A1515010396).
The third author was partially supported by the Postdoctoral Fellowship Program of CPSF
under Grant Number GZC20233474.

\addcontentsline{toc}{section}{\textbf{References}}

\bibliographystyle{alpha}

\bibliography{SAQuadF}

\

Yong HU

\medskip

Department of Mathematics

Southern University of Science and Technology


Shenzhen 518055, China


Email: huy@sustech.edu.cn

\

Jing LIU

\medskip

Department of Mathematics

Southern University of Science and Technology


Shenzhen 518055, China


Email: 12131230@mail.sustech.edu.cn

\

Yisheng TIAN

\medskip

Institute for Advanced Study in Mathematics

Harbin Institute of Technology


Harbin 150001, China


Email: tysmath@mail.ustc.edu.cn


\end{document}